\documentclass[9pt,twocolumn,twoside]{pnas-new}

\templatetype{pnasresearcharticle} 

\usepackage{amsmath}
\usepackage{amstext}
\usepackage{amsthm}
\usepackage{amssymb}
\usepackage{microtype}

\newtheorem{theorem}{Theorem}

\newtheorem{lemma}{Lemma}
\newtheorem{remark}{Remark}
\newtheorem{definition}{Definition}
\setboolean{displaywatermark}{false}

\title{When are Bayesian model probabilities overconfident?}

\author[a]{Oscar Oelrich}
\author[b]{Shutong Ding}
\author[c]{Måns Magnusson}
\author[c]{Aki Vehtari}
\author[a,d]{Mattias Villani} 

\affil[a]{Department of Statistics, Stockholm University}
\affil[b]{Department of Statistics, International Labour Organization}
\affil[c]{Department of Computer Science, Aalto University}
\affil[d]{Department of Computer and Information Science, Linköping University}

\leadauthor{Oelrich} 

\significancestatement{Bayesian model inference holds a central place in a wide range of scientific fields. However, Bayesian model probabilities often behave paradoxically in that they tend to concentrate almost entirely on a single model even when other measures of model fit do not indicate large differences between models. We document this overconfidence in high-profile applications in macroeconomics and neuroscience, and characterize the sources of overconfidence in a regression setting.}

\authorcontributions{M.V. proposed the initial idea for the paper and designed research. M.V., O.O. and S.D. proved the mathematical results. M.V., M.M., O.O. and S.D. wrote the computer code, analyzed data and conducted the experiments. M.V., O.O. and S.D. wrote the initial draft of the paper. All authors edited the paper and contributed in discussions.}
\authordeclaration{The authors declare no conflict of interest.}
\equalauthors{\textsuperscript{1}M.V., O.O. and S.D. contributed equally to this work.}
\correspondingauthor{\textsuperscript{2}To whom correspondence should be addressed. E-mail: oscar.oelrich@stat.su.se}


\begin{abstract}
Abstract. Bayesian model comparison is often based on the posterior distribution over the set of compared models. This distribution is often observed to concentrate on a single model even when other measures of model fit or forecasting ability indicate no strong preference. Furthermore, a moderate change in the data sample can easily shift the posterior model probabilities to concentrate on another model. We document overconfidence in two high-profile applications in economics and neuroscience. To shed more light on the sources of overconfidence we derive the sampling variance of the Bayes factor in univariate and multivariate linear regression. The results show that overconfidence is likely to happen when i) the compared models give very different approximations of the data-generating process, ii) the models are very flexible with large degrees of freedom that are not shared between the models, and iii) the models underestimate the true variability in the data.
\end{abstract}

\dates{This manuscript was compiled on \today}

\begin{document}

\maketitle
\thispagestyle{firststyle}
\ifthenelse{\boolean{shortarticle}}{\ifthenelse{\boolean{singlecolumn}}{\abscontentformatted}{\abscontent}}{}

\dropcap{B}ayesian inference has gained widespread popularity in recent decades,
largely propelled by advances in computing power and efficient simulation
algorithms like Markov Chain Monte Carlo \cite{brooks2011handbook}
and Sequential Monte Carlo \cite{doucet2000sequential}. The Bayesian
approach to model comparison is theoretically attractive and the standard
in many fields in both the natural and the social sciences. Some examples
are \cite{griffiths2004finding} in Linguistics, \cite{Smets2007}
in Economics, and \cite{stephan2009bayesian} in Neuroscience. Hypothesis
testing is a special case of model selection. Problems with
classical hypothesis testing and the so called reproducibility crisis
has directed attention to Bayesian model selection as an alternative,
see for example \cite{raftery1995bayesian} and \cite{johnson2013revised}.

A posterior distribution over a set of models makes it
straightforward to select one of the models for further study, or to
average inference across the models using Bayesian Model Averaging
\citep[BMA,][]{Hoeting1999}. Bayesian model comparison based on posterior model probabilities has many attractive
properties: i) it allows the compared models to be non-nested, ii)
it is consistent when the data generating process is among the compared
models (the $\mathcal{M}$-closed perspective in \citealp{Bernardo1994}),
iii) it will asymptotically concentrate the posterior probability
mass on the model closest to the data generating process when all
compared models are misspecified  ($\mathcal{M}$-open perspective in \citealp{Bernardo1994}), and iv) it has  direct connections to out-of-sample
forecasting performance and cross-validation \cite{Gelfand1994,Bernardo1994,geweke1999using,Vehtari+Ojanen:2012,fong2019marginal}.

However, Bayesian model probabilities sometimes behave puzzling in practise
in that the posterior model distribution often concentrates entirely
on one model, giving the impression of overwhelming support for that
model. At the same time other forms of model comparison, e.g. predictive
performance on a test set, do not show nearly the same degree of discrimination.
This overconfidence is part of the folklore among expert Bayesians,
but remains largely undocumented in the scientific literature, barring
brief passages such as in \cite{Yao2018} or \cite{li2019comparing} who state 'In practice
we have observed a tendency of BMA to be over confident in weighting
models\textemdash assigning weights that are too close to zero or
one'. Moreover, theoretical work on overconfidence in Bayesian model
probabilities is scarce. Two recent exceptions are \cite{Yang2018a} and \cite{huggins2019using}. The asymptotic behavior of Bayesian model comparison when
the compared models are equally misspecified is explored in \cite{Yang2018a}, showing random walk
like behavior of the log Bayes factor in large samples. The overconfidence of Bayesian posteriors and Bayesian model probabilities is also highlighted by \cite{huggins2019using} who use bagging of posteriors to make Bayesian inference more robust.

Our paper sheds light on the sources of overconfidence of Bayesian model probabilities
by deriving the sampling variance of the Bayes factor in linear regression.
We show that overconfidence is likely to be happen when i) the compared
models give very different approximations of the data-generating
process, ii) the models are very flexible, i.e. have large
degrees of freedom, and that complexity is not shared between the
models, and iii) the models are unable to replicate the variability in the data generating process. We also extend the results to multivariate regression.

The next section motivates our study by showing disturbingly clear signs of overconfidence in a high-profile
applications in macroeconomics \cite{Smets2007} and neuroscience \cite{leff2008cortical}. The rest of the paper studies the sources of overconfidence mathematically through the between-sample variance of the Bayes factor for Bayesian linear regression models. Proofs of the results are given in the Supplementary material.

\section*{Background and motivation}

\subsection*{Bayesian model probabilities}

Consider comparing a set of $K$ models, $\mathcal{M}=\left\{ M_{1},...,M_{K}\right\} $,
for the observed data $y=(y_{1},...,y_{n})^{T}$, each depending
on a vector of model parameters $\theta_{k}$ . A common way of doing Bayesian model comparison is to use the posterior distribution over the model
set $\mathcal{M},$
\[
p(M_{k}\vert y)\propto p(y\vert M_{k})p(M_{k}),
\]
where $p(M_{k})$ is the prior probability of model $M_{k}$, $p(y\vert M_{k})$
is the \emph{marginal likelihood}
\[
p(y\vert M_{k})=\int p(y\vert\theta_{k},M_{k})p(\theta_{k} \vert M_{k})d\theta_{k},
\]
and $p(\theta_{k}\vert M_k)$ is the prior distribution for $\theta_{k}$.
The Bayes factor for comparing model $M_{k}$ to model $M_l$
is 
\[
\mathrm{B}_{k,l}=\frac{p(y\vert M_{k})}{p(y\vert M_l)}.
\]

\subsection*{Micro-based general equilibrium models in macroeconomics}

To illustrate the effect of overconfidence in real-world applications,
we first consider a class of Dynamic Stochastic General Equilibrium \cite{herbst2015bayesian}
models widely used in economics. DSGE models are the main models used
for policy analysis and prediction at essentially every major monetary
and fiscal institution in the world. Bayesian model probabilities is the
standard tool for model comparison and selection among DSGE models \cite{herbst2015bayesian}.
The seven-variable Smets-Wouters \cite{Smets2007} model is the de
facto  starting point for most DSGE models used in practical
work. The Smets-Wouters model is a probability model for seven macroeconomic
time series using a complex micro-funded model based on optimizing
representative agents in the economy with rational expectations. The
model dynamics are driven by seven underlying latent time series shocks,
such as shocks in technology and preferences. In \cite{Smets2007},
the base version of the model is compared with eight variants that
restricts certain model parameters to known values. The posterior
model probabilities based on the marginal likelihoods from \cite{Smets2007}
is given in Table \ref{table:SmetsWoutersPMP}, showing conclusive
evidence in favor of model $M_{3}$.

\begin{table}[H]
\caption{Posterior model probabilities in the DSGE example.}
\label{table:SmetsWoutersPMP}
\centering{}%
\begin{tabular}{ccccccccc}
\toprule
Base & M1 & M2 & M3 & M4 & M5 & M6 & M7 & M8\tabularnewline
0.01 & 0.00 & 0.00 & 0.99 & 0.00 & 0.00 & 0.00 & 0.00 & 0.00\tabularnewline
\bottomrule
\end{tabular}
\end{table}

To investigate if overconfidence is a concern we approximate the sampling 
distribution of the posterior model probabilities using the circular block 
boostrap for time series \cite{politis1994stationary}. Figure \ref{figure:SmetsWoutersPMPboot} shows
results from $1000$ bootstrap replicates. The vertical bars correspond to
models and the horizontal stripes to bootstrap replicates. The bootstrap
replicates have been sorted with respect to the probabilities of the
baseline model. The colors represent the posterior model probabilities.
For example, a row where one of the columns has a stripe of dark purple
implies that the model in that column has strong support ($\mathrm{Pr}(M_{k}\vert\mathrm{Data})>0.9$),
for the given bootstrap replicate, and all other models weak support ($\mathrm{Pr}(M_{k}\vert\mathrm{Data})\leq0.1$).
Figure \ref{figure:SmetsWoutersPMPboot} shows that the conclusion
from Table \ref{table:SmetsWoutersPMP}, where model $M_{3}$ came
out as the sure winner, is far from robust. In a large fraction of
the bootstrap samples, we actually have $\mathrm{Pr}(M_{7}\vert\mathrm{Data})>0.9$;
there is also the same level of support for the base model in a non-negligible
fraction of bootstrap replicates. 

Figure \ref{figure:DSGEconclusive} displays the percentage of bootstrap replicates where there is strong
support for one of the models (the posterior model probability, PMP, is larger than 0.9, 0.95 and 0.99,
respectively), or where the evidence is inconclusive. If we take the outcome that one of the PMPs are larger
than $0.99$ as conclusive, then Figure \ref{figure:DSGEconclusive}
shows that we have conclusive evidence in $35$\% of the replicates,
but spread over 5 different models. We therefore conclude that the
model comparison in \cite{Smets2007} suffers from overconfidence
with misleadingly conclusive support for $M_{3}$. The authors of \cite{Smets2007}
seem in fact unimpressed by the strong support for model $M_{3}$
as they silently continue the remainder of the article with the analysis
of the base model.

\begin{figure}
\centering
\includegraphics[width=.95\linewidth]{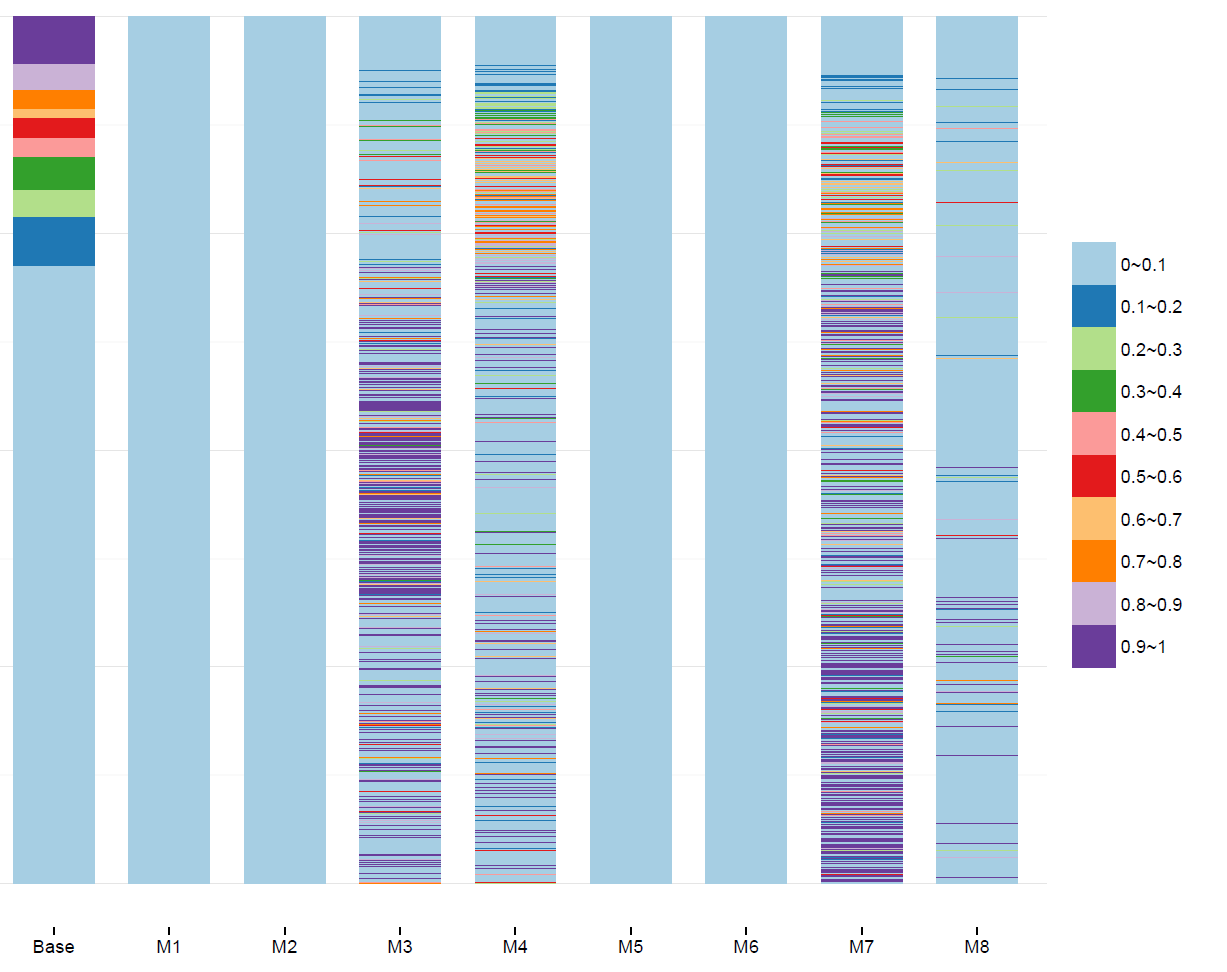}
\caption{Posterior model probabilities for the nine DSGE models in \cite{Smets2007}
for $1000$ bootstrap replicates. The bars correspond to models and
the horizontal stripes to bootstrap replicates, which have been sorted
with respect to the probabilities of the base model. The colors correspond
to posterior model probabilities. For example, a row where one of
the bars has a stripe of dark purple implies that the model corresponding to that
bar has strong support ($p(M_{k}\vert y)>0.9$)
for the given bootstrap replicate, and all other models weak support
($p(M_{k}\vert y)\protect\leq0.1$).}
\label{figure:SmetsWoutersPMPboot}
\end{figure}

\begin{figure}
\centering
\includegraphics[width=.9\linewidth]{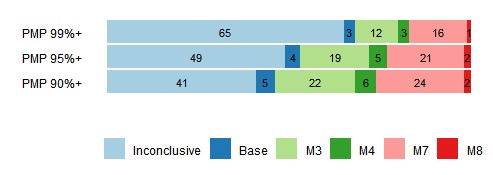}
\caption{Sampling variability in the DSGE model selection process for three posterior model probability (PMP) thresholds. The figure displays the percentage of bootstrap samples in which one model is strongly preferred, or in which the evidence is inconclusive.}
\label{figure:DSGEconclusive}
\end{figure}

\subsection*{Causal brain interactions in neuroscience}
Dynamic Causal Models (DCM) \cite{friston2003dynamic} is a popular class of models that use data from functional MRI (fMRI) brain scans to investigate how brain regions interact during an experimental task. Bayesian posterior model probabilities are the recommended method of model comparison for DCMs \cite{ashburner2014spm12}[Ch.37]. DCMs are used in \cite{leff2008cortical} to analyze how three brain regions that are known to be associated with speech processing interact when hearing intelligible speech. We reanalyze their data, but exclude subject 5 since it is a duplicate of subject 4. This leaves 25 subjects for the analysis.

Figure \ref{figure:DCMs} displays two of the compared models. A particular question of interest is in which of the three regions the auditory input is localized, regardless of the presence or absence of connectivity patterns. Such posterior probabilities are obtained by summing over all possible connections for a given source location, and is presented in Table \ref{table:DCM_PMP}, which is identical to Table 1 in \cite{leff2008cortical} even with subject 5 removed. According to Table \ref{table:DCM_PMP} we are supposed to be absolutely certain that input only enters through region P. Figure \ref{figure:DCMconclusive} however shows that in a non-negligible fraction of bootstrap samples we actually obtain conclusive evidence for region A. Figure shows the bootstrapped sampling distribution for the log Bayes factor comparing the hypothesis P against A. The regions of evidence from the well known conservative Kass-Raftery scale \citep{Kass1995} are also indicated in the figure. The sampling variance is very large, and the regions of weak or inconclusive support is but a small interval between large masses of very strong support for either of the two hypotheses. Most of the mass for strong evidence happens to be located on P in this example, but it is clear that the Bayes factor can very easily be overconfident.

\begin{figure}
\centering
\includegraphics[width=.8\linewidth]{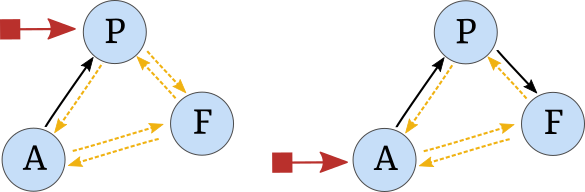}
\caption{Two of the compared DCM models for the three brain auditory brain regions: i) left posterior temporal sulcus (P), ii) left anterior superior
temporal sulcus (A) and iii) pars orbitalis of the inferior frontal gyru (F). Dashed orange arrows are endogenous task-unrelated connections which are present in all models, while solid black arrows are connections whose strength is modulated by the hearing task. The location of the auditory input is indicated by the red square with an arrow.}
\label{figure:DCMs}
\end{figure}

\begin{table}[H]
\caption{Posterior model probabilities in the DCM example.}
\label{table:DCM_PMP}
\centering{}%
\begin{tabular}{ccccccc}
\toprule
A & F & P & AF & PA & PF & PAF\tabularnewline
0.00 & 0.00 & 1.00 & 0.00 & 0.00 & 0.00 & 0.00\tabularnewline
\bottomrule
\end{tabular}
\end{table}

\begin{figure}
\centering
\includegraphics[width=.9\linewidth]{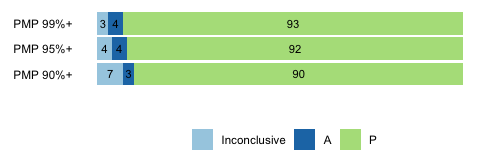}
\caption{Sampling variability in the DCM model selection process for three posterior model probability (PMP) thresholds. The figure displays the percentage of bootstrap samples in which one model is strongly preferred, or in which the evidence is inconclusive.}
\label{figure:DCMconclusive}
\end{figure}

\begin{figure}
\centering
\includegraphics[width=.9\linewidth]{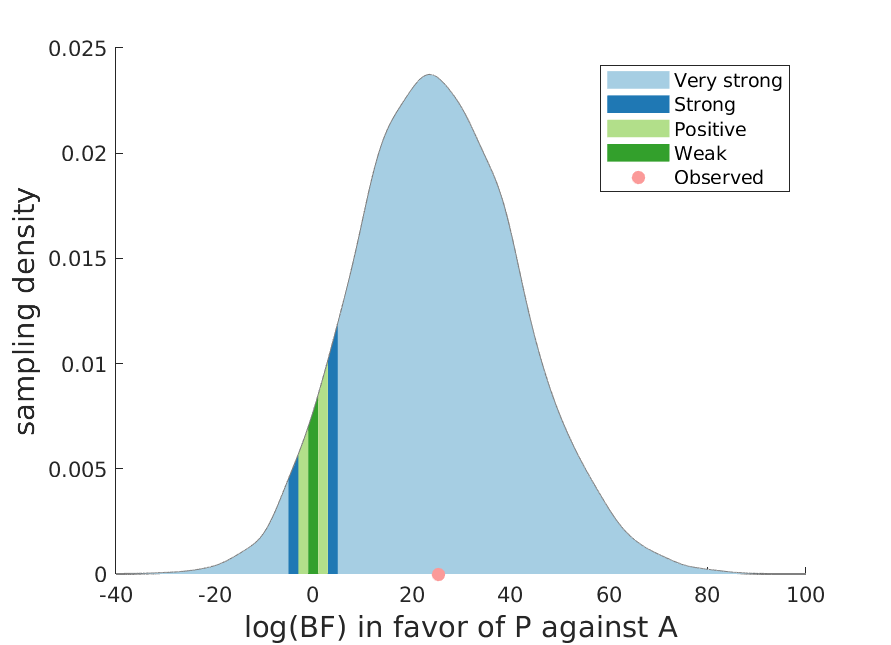}
\caption{Bootstrap approximation of the sampling distribution of $\log \mathrm{BF}_{P,A}$ comparing the family of models where the auditory input is localized only in region P against the family of models where it is localized only in region A. The sampling distribution is approximated by $10000$ bootstrap resamples of subjects. The Kass-Raftery scale of evidence is displayed by the colored regions. The observed $\log \mathrm{BF}_{P,A}$ is marked out by a pink dot.}
\label{figure:NeuroBFhistogram}
\end{figure}

\section*{Overconfidence in Gaussian linear regression models}

\subsection*{Univariate response}

We explore the sources of overconfidence of Bayesian model comparison
by deriving the sampling variance of the Bayes factor for two compared
models in a finite sample setting. To get tractable and easily interpretable results we consider Gaussian linear regression models with known error variances $\sigma_{1}^{2}$ and $\sigma_{2}^{2}$; Remark \ref{remark:unknownParameters} discusses the case with unknown variances. The data generating process
$M_{*}$ is given by

\begin{equation}
M_{*}:\qquad y=X_{*}\beta_{*}+\varepsilon_{*}\quad\varepsilon_{*}\sim \operatorname{N}(0,\sigma_{*}^{2}I_{n}),\label{eq:trueDGPLinReg}
\end{equation}
where $X_{*}$ is a $n\times p_{*}$ matrix of full rank. We compare
the misspecified models

\begin{align}
M_{1}:\qquad & y=X_{1}\beta_{1}+\varepsilon_{1}\quad\varepsilon_{1}\sim \operatorname{N}(0,\sigma_{1}^{2}I_{n})\label{eq:reg_Gaussian}\\
M_{2}:\qquad & y=X_{2}\beta_{2}+\varepsilon_{2}\quad\varepsilon_{2}\sim \operatorname{N}(0,\sigma_{2}^{2}I_{n}),\nonumber 
\end{align}
where $X_{i}$ is known and of full rank $p_{i}$, and $\beta_{i}$
is a $p_{i}\times1$ vector of unknown regression coefficients. 

We use Zellner's g-prior for both models
\begin{equation*}
\beta_{i}  \vert\sigma_{i}^{2}\sim \operatorname{N}\left(0,g\sigma_i^{2}(X_{i}^{T}X_{i})^{-1}\right)
\end{equation*}
for some shrinkage constant $g>0$, but it is straightforward to extend the results to general normal priors. The posterior is of the form

\begin{equation*}
\beta_{i}|\sigma_{i}^{2},y \sim \operatorname{N}\left(\tilde\beta_i,\kappa\sigma_i^{2}\left(X_{i}^{T}X_{i}\right)^{-1}\right),
\end{equation*}
where 

\begin{equation*}
\tilde\beta_i =\kappa\left(X_{i}^{T}X_{i}\right)^{-1}X_{i}^{T}y,
\end{equation*}
and $\kappa=\frac{g}{g+1}$ is the shrinkage factor. The Bayesian
posterior predictive mean of model $M_{i}$ is a linear smoother
\cite[Ch.~3.10]{ruppert2003semiparametric} of the form $\hat{y}_{i}=H_{i}y$, where $H_{i}=\kappa P_{i}$
is a shrunken version of the least squares projection matrix $P_{i}=X_{i}(X_{i}^{T}X_{i})^{-1}X_{i}^{T}$.

The marginal likelihood for model $M_{i}$ is given by 
\[
p(y|M_{i})=(2\pi\sigma_{i}^{2})^{{-\frac{n}{2}}}\left(1-\kappa\right)^{p_{i}/2}\exp\left\{ -\frac{1}{2\sigma_{i}^{2}}y^{T}\left(I_{n}-H_{i}\right)y\right\}.
\]

Our first result is derived under the assumption that the error variances
in the two models are known and equal. This case gives a particularly simple expression with interesting interpretation. See the Appendix for results when the model variances differ.
Let $\left\Vert x\right\Vert _{2}=(x^{T}x)^{1/2}$ be the Euclidean
norm of the vector $x$ and $\left\Vert A\right\Vert _{F}=\sqrt{\mathrm{tr}(A^{T}A)}$
the Frobenius norm of the matrix $A$.
\begin{theorem}
\label{thm:varBFLinReg}The sampling mean and variance of the log
Bayes factor for the two regression models in \textup{[\ref{eq:reg_Gaussian}]},
assuming equal and fixed variance $\sigma^{2}$, with respect to
the data-generating process in \textup{[\ref{eq:trueDGPLinReg}]} is
\begin{align*}
\mathrm{E}(\log \mathrm{B}_{12}(y)) &=\frac{\mathrm{KL}_{2}-\mathrm{KL}_{1}}{2-\kappa} +\frac{p_{1}-p_{2}}{2}\left(\log(1-\kappa)+\kappa\frac{\sigma_{*}^{2}}{\sigma^{2}}\right)\\
\mathrm{Var}(\log \mathrm{B}_{12}(y)) & =\frac{\sigma_{*}^{2}}{\sigma^{2}}\frac{\left\Vert \hat{\mu}_{1}-\hat{\mu}_{2}\right\Vert _{2}^{2}}{\sigma^{2}} + \left(\frac{\sigma_{*}^{2}}{\sqrt{2}\sigma^{2}}\left\Vert H_{1}-H_{2}\right\Vert _{F}\right)^{2},
\end{align*}
where $\hat{\mu}_{i}=H_{i}\mu_{*}$ is the projection of the true
mean vector $\mu_{*}$ onto model $M_{i}$ and $\mathrm{KL}_i$ is the Kullback-Leibler divergence of model $M_i$ with estimate $\hat\mu_i$ from the true $M_*$.
\end{theorem}

Theorem \ref{thm:varBFLinReg} shows that the Bayes factor favors models that are KL-close to the data-generating process, which is in line with the general asymptotic result in for example \cite{Berk1966,fernandez2004comparing}, but also that it penalizes complex models.

More interestingly, Theorem \ref{thm:varBFLinReg} shows that the variance increases with:
i) $\left\Vert \hat{\mu}_{1}-\hat{\mu}_{2}\right\Vert _{2}^{2}/\sigma^{2}$, 
ii) $\left\Vert H_{1}-H_{2}\right\Vert _{F}$, and iii) the variance ratio $\sigma_*^2/\sigma^2$. We discuss each of these parts in turn.

It is straightforward to prove that $\hat{\mu}_{i}$ minimizes the KL divergence
of $p(y\vert M_{i})$ from $p(y\vert M_{*})$ (see the Appendix); hence, $\hat{\mu}_{i}$
is the best approximation of  $p\left(y\vert M_{*}\right)$
that model $M_{i}$ is capable of. It is also easy to show (see the Appendix) that $\mathrm{KL}(M_{1}(\hat{\mu}_{1})\left\Vert M_{2}(\hat{\mu}_{2})\right)=\left\Vert \hat{\mu}_{1}-\hat{\mu}_{2}\right\Vert _{2}^{2}/2\sigma^{2}$, the KL divergence between
the best approximating models $\operatorname{N}(y\vert\hat{\mu}_{1},\sigma^{2})$
and $\operatorname{N}(y\vert\hat{\mu}_{2},\sigma^{2})$, which happens to be symmetric when $\sigma_1=\sigma_2$.
The first term in Theorem \ref{thm:varBFLinReg} therefore shows that
the variance of the Bayes factor tends to be large when the two models
approximate $M_{*}$ in widely different ways. This explains why continuous model expansion, where a model is embedded in a larger family via a continuous parameter, is preferred over a comparison of a discrete set of well separated models \cite{draper1995assessment}. A model is always surrounded by other similar models in continuous model expansions. We also note that the recommended strategy in \cite{draper1995assessment} is to compare widely different models that 'stake out the corners in the model space' in order to capture the true model uncertainty. This is a good strategy when one can afford to stake out the corners with a dense set of models, preferably even a continuum of models, but this is rarely the case. A much more common situation is when the model space is staked out using a small set of models. Unfortunately, Theorem \ref{thm:varBFLinReg} shows that the posterior model probabilities are then highly likely to be overconfident.

To interpret the term $\left\Vert H_{1}-H_{2}\right\Vert _{F}$
recall that the \emph{degrees of freedom} of a linear smoother, $\hat{y}=Hy$,
is given by $\mathrm{tr}(H)$ \cite{Hastie2009}. The next lemma
shows that $\left\Vert H_{1}-H_{2}\right\Vert _{F}^{2}$ is the total
degrees of freedom of the two models that is not shared between them.

\begin{lemma}
\label{lem:nonsharedComplexity}$\left\Vert H_{1}-H_{2}\right\Vert _{F}^{2}$
measures the total non-shared degrees of freedom of the models in
\textup{[\ref{eq:reg_Gaussian}]}:
\begin{equation*}
\left\Vert H_{1}-H_{2}\right\Vert _{F}^{2}  =\kappa^{2}\left(\mathrm{tr}(P_{1})+\mathrm{tr}(P_{2})-2\left(s+\Sigma_{i=1}^{r}\cos^{2}(\theta_{k+i})\right)\right),
\end{equation*}
where $p_{1}=\mathrm{tr}(P_{1})\ge\mathrm{tr}(P_{2})=p_{2}\ge1,\theta_{j}\in[0,\pi/2]$
for $j=1,...,p_{2}$ are the principal angles between $S_{1}=\mathrm{span}(H_{1})$
and $S_{2}=\mathrm{span}(H_{2})$, $s=dim\left(S_{1}\cap S_{2}\right)$
is the number of $\theta_{j}$ that are exactly zero and $r$ is the
number of $\theta_{j}$ in the open interval $\left(0,\pi/2\right).$
\end{lemma}

By Lemma \ref{lem:nonsharedComplexity}, $\left\Vert H_{1}-H_{2}\right\Vert _{F}^{2}$
is the total complexity of the two models, $\mathrm{tr}(P_{1})+\mathrm{tr}(P_{2})$,
reduced by the $s$ completely shared dimensions and by the cosine
of the principal angles of the $r$ partially shared dimensions,
shrunk by the precision of the prior. Hence, $\mathrm{Var}(\log \mathrm{B}_{12}(y))$ is not affected by any
complexity that is shared between the models. Comparison of a dense set of nested models, e.g. variable selection in regression, is therefore expected to less prone to overconfidence since there is often a large overlap between compared models. Figure \ref{fig:modelCartoon} gives an abstract illustration of models, divergences between models and their shared/non-shared complexities.

\begin{figure*}
\centering
\includegraphics[width=.23\linewidth]{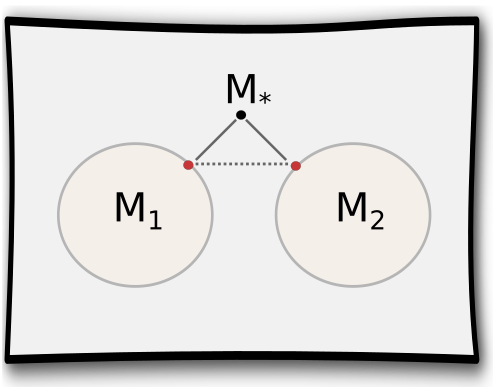}
\includegraphics[width=.23\linewidth]{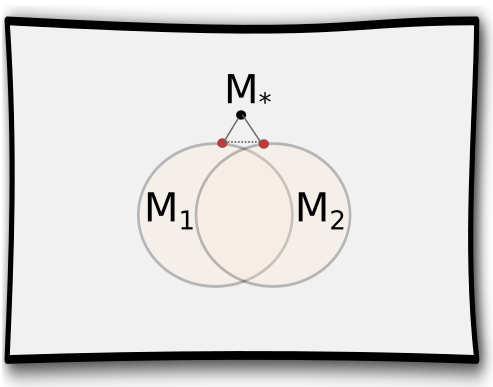}
\includegraphics[width=.23\linewidth]{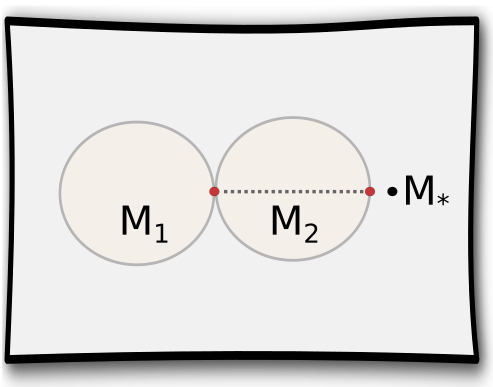}
\caption{Illustrating divergence and complexity of models. Each figure displays the data generating process ($M_*$) as a black point. The two models $M_1$ and $M_2$ are depicted as circles where the area of each circle indicates the expressiveness, or complexity, of the model, i.e. $\mathrm{tr}(H_i)$. The red points represent $\hat \mu_i$, the best approximation to $M_*$ within $M_i$. Finally, the dotted line illustrates the Kullback-Leibler divergence between the models, i.e. $\vert \vert \hat\mu_1 - \hat \mu_2 \vert \vert$. The left graph shows models with no shared complexity, the middle graph shows models with substantial shared complexity, and the graph to the right shows models without shared complexity where one of the models is closer to $M_*$.}\label{fig:modelCartoon}
\label{figure:DistanceIllustration}
\end{figure*}

Finally, both $\left\Vert \hat{\mu}_{1}-\hat{\mu}_{2}\right\Vert _{2}^{2}/\sigma^{2}$
and $\left\Vert H_{1}-H_{2}\right\Vert _{F}$ in Theorem \ref{thm:varBFLinReg} are inflated by the
error variance ratio $\sigma_{*}^{2}/\sigma^{2}$ in the expression
for $\mathrm{Var}(\log \mathrm{B}_{12}(y))$. Models that are unable to
generate enough variation in their data distribution are therefore
particularly susceptible to overconfidence.

We make the following additional remarks about Theorem \ref{thm:varBFLinReg}.

\begin{remark}
$\mathrm{Var}(\log \mathrm{B}_{12}(y))$ does not directly depend on the degree of misspecification
of the two models relative to the true $M_{*}$; only the
divergence between the models matters. However, the degree of misspecification  restricts how far apart the models can be. One way to see this is by noting that $\left\Vert \hat{\mu}_{1}-\hat{\mu}_{2}\right\Vert _{2}^{2}$ has the upper bound $$\left\Vert \hat{\mu}_{*}-\hat{\mu}_{1}\right\Vert _{2}^{2}+\left\Vert \hat{\mu}_{*}-\hat{\mu}_{2}\right\Vert _{2}^{2}+2\left\Vert \hat{\mu}_{*}-\hat{\mu}_{1}\right\Vert _{2}^{2}\left\Vert \hat{\mu}_{*}-\hat{\mu}_{2}\right\Vert _{2}^{2},$$
which under equal degree of misspecification simplifies to $$\left\Vert \hat{\mu}_{1}-\hat{\mu}_{2}\right\Vert _{2}^{2}\le4\left\Vert \hat{\mu}_{*}-\hat{\mu}_{i}\right\Vert _{2}^{2}.$$
Hence, in problems where all models are very misspecified there is greater scope for the models to approximate $p(y\vert M_*)$ in markedly different ways, and therefore greater risk of overconfidence.
\end{remark}

\begin{remark}\label{remark:unknownParameters}
Since the known variance $\sigma^2$ enters the mean and variance in Theorem \ref{thm:varBFLinReg} as a multiplicative factor $1/\sigma^2$, it is straightforward to generalize Theorem \ref{thm:varBFLinReg} to the case with a common and unknown variance by applying the law of total variance. The end result is that the factor $1/\sigma^2$ is replaced by its prior expectation, but the two main terms $\left\Vert \hat{\mu}_{1}-\hat{\mu}_{2}\right\Vert _{2}^{2}/\sigma^{2}$ and $\left\Vert H_{1}-H_{2}\right\Vert _{F}$ remain. The case with unknown and different variances seems to be intractable, but numerical experiments indicate that the same three factors are driving $\mathrm{Var}(\log \mathrm{B}_{12}(y))$.
\end{remark}

\begin{remark}
The technique behind Theorem \ref{thm:varBFLinReg} can also be used for regression models with heteroscedastic variance as the log of the marginal likelihood remains quadratic in $y$ with $I_{n}-H_{i}$ replaced by $\Sigma_{\varepsilon}^{-1/2}(I_{n}-H_{i})\Sigma_{\varepsilon}^{-1/2}$, where $\Sigma_{\varepsilon}$ is the $n\times n$ covariance matrix of errors. This framework includes the popular Gaussian process regression models in machine learning where
\begin{equation*}
y_i = x_i^T\beta + f(x),
\end{equation*}
and the function $f(x)$ follows a Gaussian process with a given covariance kernel. The log marginal likelihood is given in Equation 2.43 of \cite{rasmussen2003gaussian}.
\end{remark}

\subsection*{Multivariate response}

The previous section shows that overconfidence is a bigger concern when the compared models are approximating the data-generating process in very different ways. The DSGE models from \cite{Smets2007} are seemingly similar, however; they only differ from the baseline model by setting a single parameter to a specific value. Note however that the results in the previous section are derived for the case with a univariate response, while the DSGE models are multivariate model with seven time series responses. This subsection extends the previous results to multivariate regression and highlights some special properties for this more general case.  

The data-generating process $M_{*}$ is given by the multivariate regression

\begin{equation}
M_{*}^ {}:\qquad Y=X_{*}B_{*}+E_{*}\quad E_{*}\sim \operatorname{N}(0,I_{n},\Sigma_{*}),\label{eq:trueDGPLinReg-1}
\end{equation}
and we consider comparing the following models

\begin{align}
M_{1}^ {}:\qquad & Y=X_{1}B_{1}+E_{1}\quad E_{1}\sim \operatorname{N}(0,I_{n},\Sigma_{1})\label{eq:reg_Gaussian-1}\\
M_{2}^ {}:\qquad & Y=X_{2}B_{2}+E_{2}\quad E_{2}\sim \operatorname{N}(0,I_{n},\Sigma_{2}),\nonumber 
\end{align}
where  $\operatorname{N}(\mu,A_1,A_2)$ denotes the matrix variate distribution, $Y$ is $n\times q$ , $X_{i}$ is known and of full 
rank $p_{i}$, and $B_{i}$ is a $p_{i}\times q$ vector of unknown
parameters. The error term $E_{i}=(\varepsilon_{1i},\varepsilon_{2i},...,\varepsilon_{ni})'$
is an $n\times q$ matrix, following a matrix normal distribution, with rows that are iid $\operatorname{N}(0,\Sigma_i)$.
We let \textbf{$\beta_{i}=\mathrm{vec}(B_{i})$}, and use Zellner's g-prior
for the regression coefficients.

\begin{theorem}
\label{thm:varBFLinReg-1}The sampling mean and variance of the log
Bayes factor for the two multivariate regression models in \textup{[\ref{eq:reg_Gaussian-1}]}
with respect to the data-generating process in \textup{[\ref{eq:trueDGPLinReg-1}]}
is
\begin{align*}
\mathrm{E}(\log \mathrm{B}_{12}(y)) & =\frac{\mathrm{KL}_2-\mathrm{KL}_1}{2-\kappa}\\
&\quad+\frac{p_{1}-p_{2}}{2}\left(\log(1-\kappa)+\kappa\mathrm{tr}\left(\Sigma^{-1}\Sigma_{*}\right)\right)\\
\mathrm{Var}(\log \mathrm{B}_{12}(y)) & =\frac{1}{2}\mathrm{tr}\left(\Omega^{2}\right)\left\Vert H_{2}-H_{1}\right\Vert _{F}^{2}\\
 & \quad+\left\Vert (\hat{\mu}_{2}-\hat{\mu}_{1})\Sigma^{-1/2}\Omega^{1/2}\right\Vert _{F}^{2},
\end{align*}
where $\Omega\equiv\Sigma^{-1/2}\Sigma_{*}\Sigma^{-1/2}$ is a multivariate
generalization of the variance ratio $\sigma_{*}^{2}/\sigma^{2}$ and $\hat{\mu}_{i}=H_{i}\mu_{*}$ is the projection of the true mean vector $\mu_{*}$ onto model $M_{i}$.
\end{theorem}

The interpretation remains largely the same as in the univariate case,
with the added insight that not all differences in the models are
equally important due to the appearance of the generalized variance ratio $\Omega$ in $\left\Vert (\hat{\mu}_{2}-\hat{\mu}_{1})\Sigma^{-1/2}\Omega^{1/2}\right\Vert _{F}^{2}$. To show this more precisely, we perform the two spectral decompositions:
\begin{align*}
\Sigma & =U\Lambda U^{T}=\left(U\Lambda^{1/2}\right)\left(U\Lambda^{1/2}\right)^{T}\\
\Sigma_{*} & =U_{*}\Lambda_{*}U_{*}^{T}=\left(U_{*}\Lambda_{*}^{1/2}\right)\left(U_{*}\Lambda_{*}^{1/2}\right)^{T},
\end{align*}
where $U=(u_{1},...,u_{p})$ and $U_{*}=(u_{*1},...,u_{*p})$ are
matrices of eigenvectors, and $\Lambda=\mathrm{Diag}(\lambda_{1},...,\lambda_{p})$
and $\Lambda_{*}=\mathrm{Diag}(\lambda_{*1},...,\lambda_{*p})$ are
diagonal matrices of eigenvalues. Now, $\Sigma^{-1/2}=U\Lambda^{-1/2}$
and $\Omega=\Lambda^{-1/2}U^{T}U_{*}\Lambda_{*}U_{*}^{T}U\Lambda^{-1/2}$, therefore
\[
(\hat{\mu}_{2}-\hat{\mu}_{1})\Sigma^{-1/2}\Omega^{1/2}=(\hat{\mu}_{2}-\hat{\mu}_{1})U\Lambda^{-1/2}\left(\Lambda^{-1/2}U^{T}U_{*}\Lambda_{*}^{1/2}\right).
\]
The $n\times p$ matrix $\left(\hat{\mu}_{1}-\hat{\mu}_{2}\right)U\Lambda^{-1/2}$
contains the differences in predictions in the directions of the principal
components of $\Sigma$, rescaled to unit variance. The $p\times p$
matrix $\Lambda^{-1/2}U^{T}U_{*}\Lambda_{*}^{1/2}$ has element $\sqrt{\lambda_{*j}/\lambda_{i}}\cos(\phi_{u_{i},u_{*j}})$
in its $i$th row, $j$th column, where $\cos(\phi_{u_{i},u_{*j}})=u_{i}^{T}u_{*j}$
measures the degree of alignment of pairs of eigenvectors from $\Sigma$ and $\Sigma_{*}$. Hence,
$\mathrm{var}\log B_{12}(Y)$ will be large when the models make very different
prediction on linear combinations of response variables where the eigenvectors
of $\Sigma$ and $\Sigma_{*}$ align, and the variance ratio $\lambda_{*j}/\lambda_{i}$
is large. This agrees with the analysis of forecasting performance in \cite{adolfson2007forecasting} for an open-economy extension of the Smets-Wouters model. In addition, \cite{adolfson2007forecasting} show that a multivariate measure of out-of-sample forecasting performance is almost entirely driven by forecasting errors in employment, one of the least important variables from a central bank perspective.

\section*{Discussion and Conclusion}

We have demonstrated that Bayesian posterior model probabilities can be overconfident in the sense of spuriously picking out one of the compared models as the only probable model in a set of compared models, while at the same time being equally certain about another model in a slightly different dataset. We have analyzed the sources of this overconfidence by deriving the sample variance of the log Bayes factor for univariate and multivariate regression. 

The main message is that overconfidence is likely to be a problem when the compared models give very different approximations of the data-generating process and when the compared models are flexible in a way that is not shared among the models. The same is true for the multivariate setting, with the added nuance that overconfidence will be largest when the models are different with respect to specific linear combinations of the response variables.

Our results motivate several interesting avenues for future research. First, the linear regression setup was chosen since it provides a clear view of what drives overconfidence. It would be illuminating to derive similar measures for more general models to see if the same effects appear there. Second, it would of interest to repeat our analysis for other Bayesian model inference methods, for example prediction pools or stacking \cite{Geweke2012,Yao2018}. Stacking is particularly interesting since it is known to not necessarily concentrate on a single model when the sample grows large. Third, we have used the bootstrap to approximate the sampling distribution of Bayes factors. This can be very time-consuming when models are analyzed by Markov Chain Monte Carlo (MCMC), since we would have to run the MCMC for each bootstrap sample. It would therefore be of practical importance to explore the efficiency of methods where the marginal likelihood for each bootstrap sample is obtained by reweighting the posterior draws from a single MCMC run on the original dataset \cite{Geweke1999}.

\begin{figure*}
\section{Appendix A: Proofs} \label{sec_The_Between_Sample}

\subsection*{Preliminaries}
The following lemma about Gaussian quadratic forms will be useful.
\begin{lemma}
Let $A_{1}$ and $A_{2}$ be $n\times n$ symmetric, nonstochastic
matrices, and let $y$ be an $n\times1$ vector following a normal
distribution, $y\sim \operatorname{N}(\mu,\Sigma).$ Then \label{lem:quad_cov}
\begin{enumerate}
\item[i)] \setlength{\itemsep}{0.5mm}$\mathrm{E}(y^{T}A_{1}y)=\mu^{T}A_{1}\mu+\mathrm{tr}(A_{1}\Sigma);$
\item[ii)] $\mathrm{Var}(y^{T}A_{1}y)=2\mathrm{tr}((A_{1}\Sigma)^{2})+4\mu^{T}A_{1}\Sigma A_{1}\mu;$
\item[iii)] $\mathrm{Cov}(y^{T}A_{1}y,y^{T}A_{2}y)=2\mathrm{tr}(A_{1}\Sigma A_{2}\Sigma)+4\mu^{T}A_{1}\Sigma A_{2}\mu.$
\end{enumerate}
\end{lemma} 
\begin{proof}
Part i) is a standard result and ii) follows from iii). To derive the covariance in iii) note that
$$\mathrm{Cov}(y^{T}A_{1}y,y^{T}A_{2}y) = \mathrm{E}\big((y^{T}A_{1}y)(y^{T}A_{2}y)\big) - \mathrm{E}(y^{T}A_{1}y)\mathrm{E}(y^{T}A_{2}y)$$
and
$$\mathrm{E}\big((y^{T}A_{1}y)(y^{T}A_{2}y)\big) = \mathrm{E}\big((z^{T}B_{1}z)(z^{T}B_{2}z)\big),$$ 
where $z = \Sigma^{-1/2}y \sim \mathrm{N}(\mu_z,I)$, $\mu_z = \Sigma^{-1/2}\mu$ and $B_{j}=\Sigma^{1/2}A_{j}\Sigma^{1/2}$. From Theorem 1 in \cite{bao2010expectation} we have 
$$\mathrm{E}\big((z^{T}B_{1}z)(z^{T}B_{2}z)\big) = E(z^{T}B_{1}z)E(z^{T}B_{2}z)+2\mathrm{tr}(B_{1}B_{2})+4\mu_{z}^{T}B_{1}B_{2}\mu_{z},$$
from which Lemma 2 iii) follows.
\end{proof}

\subsection*{Proof of Theorem \ref{thm:varBFLinReg}}
The log Bayes factor is equal to
\[
\frac{n}{2}\log\left(\frac{\sigma_{2}^{2}}{\sigma_{1}^{2}}\right)+\frac{p_{1}-p_{2}}{2}\log(1-\kappa)+Q_{2}(y)-Q_{1}(y),
\]
where $Q_{i}(y)=\frac{1}{2\sigma_{i}^{2}}y^{T}(I_n-H_{i})y$ and $H_{i}=\kappa P_{i}$
with  $P_{i}=X_{i}\left(X_{i}^{T}X_{i}\right)^{-1}X_{i}^{T}$ being a symmetric and idempotent projection matrix.
Using Part $i$ of Lemma \ref{lem:quad_cov}, we have
\begin{align*}
\mathrm{E}(\log\mathrm{B}_{12}(y))= & \frac{n}{2}\log\left(\frac{\sigma_{2}^{2}}{\sigma_{1}^{2}}\right)+\frac{p_{1}-p_{2}}{2}\log(1-\kappa) +\frac{\sigma_{*}^{2}(n-\kappa p_{2})}{2\sigma_{2}^{2}}-\frac{\sigma_{*}^{2}(n-\kappa p_{1})}{2\sigma_{1}^{2}}\\
 & +\frac{1}{2\sigma_{2}^{2}}\mu_{*}^{T}(I_{n}-H_{2})\mu_{*}-\frac{1}{2\sigma_{1}^{2}}\mu_{*}^{T}(I_{n}-H_{1})\mu_{*}
\end{align*}
since $\mathrm{tr}(\sigma_{*}^{2}(I_{n}-H_{i}))=\sigma_{*}^{2}(n-\kappa p_{i})$.
The term $\mu_{*}^{T}(I_{n}-H_{i})\mu_{*}$ can be shown to be a linearly increasing function of the Kullback-Leibler divergence of $M_{i}$ with the ideal estimate $\hat{\mu}_{i}=H_{i}\mu_{*}$ from $M_{*}$:
\begin{equation*}
\mathrm{KL}(M_{*}\vert\vert M_{i})  \equiv\int\log\left(\frac{p(y\vert\mu_{*})}{p(y\vert\hat{\mu}_{i})}\right)p(y\vert\mu_{*})dy =-\frac{n}{2}\log\left(\frac{\sigma_{*}^{2}}{\sigma_{i}^{2}}\right)+\frac{n}{2}\frac{\sigma_{*}^{2}}{\sigma_{i}^{2}}+\frac{1}{2}\frac{\left\Vert \mu_{*}-\hat{\mu}_{i}\right\Vert _{2}^{2}}{\sigma_{i}^{2}}-\frac{n}{2},
\end{equation*}
using the Kullback-Leibler divergence between two multivariate
normal densities \cite{cover2012elements}. To show the exact connection between $\mathrm{E}(\log\mathrm{B}_{12}(y))$ and $\mathrm{KL}(M_{*}\vert\vert M_{2})-\mathrm{KL}(M_{*}\vert\vert M_{1})$ we will here consider the algebraically less involved special case $\sigma_{1}=\sigma_{2}$.
For this case, the difference in KL divergences simplifies to
\[
\mathrm{KL}(M_{*}\vert\vert M_{2})-\mathrm{KL}(M_{*}\vert\vert M_{1})=
\frac{1}{2}\frac{\left\Vert \mu_{*}-\hat{\mu}_{1}\right\Vert _{2}^{2}}{\sigma^{2}}
-\frac{1}{2}\frac{\left\Vert \mu_{*}-\hat{\mu}_{2}\right\Vert _{2}^{2}}{\sigma^{2}}.
\]
Note that
\begin{equation*}
\left\Vert \mu_{*}-\hat{\mu}_{i}\right\Vert _{2}^{2} =\mu_{*}^{T}\mu_{*}+\kappa^{2}\mu_{*}^{T}P_{i}\mu_{*}^{T}-2\kappa\mu_{*}^{T}P_{i}\mu_{*}^{T}=\mu_{*}^{T}\mu_{*}-\kappa(2-\kappa)\mu_{*}^{T}P_{i}\mu_{*}^{T}
\end{equation*}

so
\[
\mu_{*}^{T}P_{i}\mu_{*}^{T}=\frac{\mu_{*}^{T}\mu_{*}-\left\Vert \mu_{*}-\hat{\mu}_{i}\right\Vert _{2}^{2}}{\kappa(2-\kappa)}.
\]
When $\sigma_{1}=\sigma_{2}$ we therefore have that
\begin{align*}
\mathrm{E}(\log\mathrm{B}_{12}(y)) &=\frac{p_{1}-p_{2}}{2}\left(\log(1-\kappa)+\kappa\frac{\sigma_{*}^{2}}{\sigma^{2}}\right)+\frac{\kappa}{2\sigma^{2}}( \mu_{*}^{T}P_{1}\mu_{*}-\mu_{*}^{T}P_{2}\mu_{*})\\
 & =\frac{p_{1}-p_{2}}{2}\left(\log(1-\kappa)+\kappa\frac{\sigma_{*}^{2}}{\sigma^{2}}\right)
 +\frac{1}{2-\kappa}\left(\mathrm{KL}(M_{*}\vert\vert M_{2})-\mathrm{KL}(M_{*}\vert\vert M_{1})\right).
\end{align*}
\end{figure*}

\begin{figure*}
By Part ii and iii of Lemma \ref{lem:quad_cov} the variance is
\begin{align*}
\mathrm{Var}(\log\mathrm{B}_{12}(y)) &={ { \mathrm{Var}(Q_{2}(y))+\mathrm{Var}(Q_{1}(y))-2\mathrm{Cov}(Q_{2}(y),Q_{1}(y))}}\\
&  =\frac{\sigma_{*}^{4}}{2\sigma_{2}^{4}}\mathrm{tr}((I_n-H_{2})^{2})+\frac{\sigma_{*}^{2}}{\sigma_{2}^{4}}\mu_{*}^{T}(I_n-H_{2})^{2}\mu_{*}
+\frac{\sigma_{*}^{4}}{2\sigma_{1}^{4}}\mathrm{tr}((I_n-H_{1})^{2})+\frac{\sigma_{*}^{2}}{\sigma_{1}^{4}}\mu_{*}^{T}(I_n-H_{1})^{2}\mu_{*}\\
&  \quad-\frac{\sigma_{*}^{4}}{\sigma_{1}^{2}\sigma_{2}^{2}}\mathrm{tr}((I_n-H_{2})(I_n-H_{1}))-\frac{2\sigma_{*}^{2}}{\sigma_{1}^{2}\sigma_{2}^{2}}\mu_{*}^{T}(I_n-H_{2})(I_n-H_{1})\mu_{*}\\
&  =\frac{\sigma_{*}^{4}}{2}\mathrm{tr}\left(\left(\frac{\left(I_n-H_{2}\right)}{\sigma_{2}^{2}}-\frac{\left(I_n-H_{1}\right)}{\sigma_{1}^{2}}\right)^{2}\right)+\sigma_{*}^{2}\left(\mu_{*}^{T}\left(\frac{\left(I_n-H_{2}\right)}{\sigma_{2}^{2}}-\frac{\left(I_n-H_{1}\right)}{\sigma_{1}^{2}}\right)^{2}\mu_{*}\right)\\
& =\frac{\sigma_{*}^{4}}{2}\mathrm{tr}\left(\left(\frac{\left(I_n-H_{2}\right)}{\sigma_{2}^{2}}-\frac{\left(I_n-H_{1}\right)}{\sigma_{1}^{2}}\right)^{2}\right)\\
&  \quad +\sigma_{*}^{2}\left(\frac{1}{\sigma_{2}^{2}}\left(\mu_{*}-\hat{\mu}_{2}\right)-\frac{1}{\sigma_{1}^{2}}\left(\mu_{*}-\hat{\mu}_{1}\right)\right)^{T}\left(\frac{1}{\sigma_{2}^{2}}\left(\mu_{*}-\hat{\mu}_{2}\right)-\frac{1}{\sigma_{1}^{2}}\left(\mu_{*}-\hat{\mu}_{1}\right)\right).
\end{align*}

Assuming that the error variances of the two misspecified models are
equal, the expression simplifies to
\begin{align*}
\mathrm{Var}(\log\mathrm{B}_{12}(y)) & { =\frac{\sigma_{*}^{4}}{2\sigma^{4}}\mathrm{tr}\left(\left(H_{1}-H_{2}\right)^{2}\right)+\frac{\sigma_{*}^{2}}{\sigma^{4}}\left(\mu_{*}^{T}\left(H_{1}-H_{2}\right)^{2}\mu_{*}\right)}\\
 & { =\left(\frac{\sigma_{*}^{2}}{\sqrt{2}\sigma^{2}}\left\Vert H_{1}-H_{2}\right\Vert _{F}\right)^{2}+\left(\frac{\sigma_{*}}{\sigma}\frac{\left\Vert \hat{\mu}_{1}-\hat{\mu}_{2}\right\Vert _{2}}{\sigma}\right)^{2}}.
\end{align*}
\end{figure*}

\begin{figure*}
\subsection{Proof of Theorem \ref{thm:varBFLinReg-1}}\label{sec:The-Between-Sample-Multi}
We derive the result for the case $\Sigma_{1}=\Sigma_{2}$. The proof
when $\Sigma_{1}\ne\Sigma_{2}$ is similar, but more tedious algebraically.
The log Bayes factor is

\begin{align*}
\log\mathrm{B}_{12}(y) & =
\frac{p_{1}-p_{2}}{2}\log(1-\kappa)  -\frac{1}{2}\mathrm{tr}\left(Y^{T}\left(I_n-H_{1}\right)Y\Sigma^{-1}\right)+\frac{1}{2}\mathrm{tr}\left(Y^{T}\left(I_n-H_{2}\right)Y\Sigma^{-1}\right)\\
& =\frac{p_{1}-p_{2}}{2}\log(1-\kappa)  -\frac{1}{2}\mathrm{tr}\left(Y^{T}\left(H_{2}-H_{1}\right)Y\Sigma^{-1}\right)\\
& =\frac{p_{1}-p_{2}}{2}\log(1-\kappa)-  \frac{1}{2}\mathrm{vec}(Y)^{T}\left(\Sigma^{-1}\otimes\left(H_{2}-H_{1}\right)\right)\mathrm{vec}(Y),
\end{align*}
where $\mathrm{vec}(Y)\sim \operatorname{N}\left(\mathrm{vec}(\mu_{*}),\Sigma_{*}\otimes I_{n}\right)$,
$\mu_{*}=X_{*}B_{*}$, and $H_{i}=\kappa P_{i}$ with $P_{i}=X_{i}\left(X_{i}^{T}X_{i}\right)^{-1}X_{i}^{T}$.
Using Lemma \ref{lem:quad_cov}, we obtain the mean as
\begin{align*}
\mathrm{E}(\log \mathrm{B}_{12}(Y)) & =\frac{p_{1}-p_{2}}{2}\log(1-\kappa)-\frac{1}{2}(\mathrm{vec}(\mu_{*}))^{T}\left(\Sigma^{-1}\otimes\left(H_{2}-H_{1}\right)\right)\mathrm{vec}(\mu_{*})-\frac{1}{2}\mathrm{tr}\left(\left(\Sigma^{-1}\otimes\left(H_{2}-H_{1}\right)\right)\Sigma_{*}\otimes I_{n}\right)\\
 & =\frac{p_{1}-p_{2}}{2}\log(1-\kappa)-\frac{1}{2}\mathrm{tr}\left(\mu_{*}^{T}\left(H_{2}-H_{1}\right)\mu_{*}\Sigma^{-1}\right)-\frac{1}{2}\mathrm{tr}\left(\Sigma^{-1}\Sigma_{*}\otimes\left(H_{2}-H_{1}\right)\right)\\
 & =\frac{p_{1}-p_{2}}{2}\log(1-\kappa)+\frac{1}{2}\mathrm{tr}\left(\mu_{*}^{T}H_{1}\mu_{*}\Sigma^{-1}\right)-\frac{1}{2}\mathrm{tr}\left(\mu_{*}^{T}H_{2}\mu_{*}\Sigma^{-1}\right)-\frac{1}{2}\mathrm{tr}\left(\Sigma^{-1}\Sigma_{*}\right)\mathrm{tr}\left(H_{2}-H_{1}\right)\\
 & =\frac{p_{1}-p_{2}}{2}\log(1-\kappa)+\frac{1}{2}\mathrm{tr}\left(\mu_{*}^{T}H_{1}\mu_{*}\Sigma^{-1}\right)-\frac{1}{2}\mathrm{tr}\left(\mu_{*}^{T}H_{2}\mu_{*}\Sigma^{-1}\right)-\frac{1}{2}\mathrm{tr}\left(\Sigma^{-1}\Sigma_{*}\right)\kappa\left(p_{2}-p_{1}\right)
\end{align*}
Similarly to the proof of the univariate case in Theorem \ref{thm:varBFLinReg}, we can relate the terms $\mathrm{tr}\left(\mu_{*}^{T}H_{i}\mu_{*}\Sigma^{-1}\right)$
to the Kullback-Leibler divergence between $p(Y\vert\mu_{*})$ and
$p(Y|\hat{\mu}_{i})$. The Kullback-Leibler divergence of $p(Y|\hat{\mu}_{i})$ from $p(Y\vert\mu_{*})$ can be derived by employing the usual KL
divergence between multivariate normals in \cite{cover2012elements} to $p(\mathrm{vec}(Y)\vert\mathrm{vec}(\mu_{*}))$
and $p(\mathrm{vec}(Y)|\mathrm{vec}(\hat{\mu}_{i}))$ to obtain

\begin{align*}
\mathrm{KL}(M_{*}\vert\vert M_{i}) & =\frac{1}{2}\left(\log\frac{\left|\Sigma_{i}\otimes I_{n}\right|}{\left|\Sigma_{*}\otimes I_{n}\right|}-nd+\mathrm{tr}\left((\Sigma_{i}^{-1}\otimes I_{n})(\Sigma_{*}\otimes I_{n})\right)+(\mathrm{vec}(\mu_{*})-\mathrm{vec}(\hat{\mu}_{i}))^{T}(\Sigma_{i}^{-1}\otimes I_{n})(\mathrm{vec}(\mu_{*})-\mathrm{vec}(\hat{\mu}_{i}))\right)\\
 & =\frac{1}{2}\left(n\log\frac{\left|\Sigma_{i}\right|}{\left|\Sigma_{*}\right|}-nd+n\mathrm{tr}\left(\Sigma_{i}^{-1}\Sigma_{*}\right)+\mathrm{tr}\left((\mu_{*}-\hat{\mu}_{i})^{T}(\mu_{*}-\hat{\mu}_{i})\Sigma_{i}^{-1}\right)\right)\\
 & =\frac{1}{2}\left(-n\log\frac{\left|\Sigma_{*}\right|}{\left|\Sigma_{i}\right|}-nd+n\mathrm{tr}\left(\Sigma_{i}^{-1}\Sigma_{*}\right)+\left\Vert (\mu_{*}-\hat{\mu}_{i})\Sigma_{i}^{-1/2}\right\Vert _{F}^{2}\right),
\end{align*}
where $\left\Vert (\mu_{*}-\hat{\mu}_{i})\Sigma_{i}^{-1/2}\right\Vert _{F}^{2}$
is the Mahalanobis distance with respect to $\Sigma_{i}$. Analogous calculations as in the univariate case gives
\[
\mathrm{tr}\left(\mu_{*}^{T}H_{i}\mu_{*}\Sigma^{-1}\right)=\frac{\mathrm{tr}\left(\mu_{*}^{T}\mu_{*}\Sigma^{-1}\right)-\left\Vert (\mu_{*}-\hat{\mu}_{i})\Sigma_{i}^{-1/2}\right\Vert _{F}^{2}}{2-\kappa}
\]
and hence we can express the mean as
\begin{align*}
\mathrm{E}(\log \mathrm{B}_{12}(Y)) & =\frac{p_{1}-p_{2}}{2}\log(1-\kappa)+\frac{\frac{1}{2}\left\Vert (\mu_{*}-\hat{\mu}_{2})\Sigma_{2}^{-1/2}\right\Vert _{F}^{2}-\frac{1}{2}\left\Vert (\mu_{*}-\hat{\mu}_{1})\Sigma_{1}^{-1/2}\right\Vert _{F}^{2}}{2-\kappa}-\frac{1}{2}\mathrm{tr}\left(\Sigma^{-1}\Sigma_{*}\right)\kappa\left(p_{2}-p_{1}\right)\\
 & =\frac{\mathrm{KL}(M_{*}\vert\vert M_{2})-\mathrm{KL}(M_{*}\vert\vert M_{1})}{2-\kappa}+\frac{p_{1}-p_{2}}{2}\left(\log(1-\kappa)+\kappa\mathrm{tr}\left(\Sigma^{-1}\Sigma_{*}\right)\right).
\end{align*}

\begin{align*}
\mathrm{Var}(\log \mathrm{B}_{12}(Y)) & =\frac{1}{2}\mathrm{tr}\left\{ \left(\Sigma^{-1}\otimes(H_{2}-H_{1})\right)\left(\Sigma_{*}\otimes I_{n}\right)\right\} ^{2}+\mathrm{vec}(\mu_{*})^{T}\left(\Sigma^{-1}\otimes(H_{2}-H_{1})\right)\left(\Sigma_{*}\otimes I_{n}\right)\left(\Sigma^{-1}\otimes(H_{2}-H_{1})\right)\mathrm{vec}(\mu_{*})\\
 & =\frac{1}{2}\mathrm{tr}\left(\Sigma^{-1}\Sigma_{*}\Sigma^{-1}\Sigma_{*}\otimes(H_{2}-H_{1})^{2}\right)+\mathrm{vec}(\mu_{*})^{T}\left(\Sigma^{-1}\Sigma_{*}\Sigma^{-1}\otimes(H_{2}-H_{1})^{2}\right)\mathrm{vec}(\mu_{*})\\
 & =\frac{1}{2}\mathrm{tr}\left(\Sigma^{-1}\Sigma_{*}\Sigma^{-1}\Sigma_{*}\right)\left\Vert H_{2}-H_{1}\right\Vert _{F}^{2}+\mathrm{tr}\left(\mu_{*}^{T}(H_{2}-H_{1})^{2}\mu_{*}\Sigma^{-1}\Sigma_{*}\Sigma^{-1}\right)\\
 & =\frac{1}{2}\mathrm{tr}\left(\Omega^{2}\right)\left\Vert H_{2}-H_{1}\right\Vert _{F}^{2}+\mathrm{tr}\left(\left((\hat{\mu}_{2}-\hat{\mu}_{1})\Sigma^{-1/2}\right)^{T}\left((\hat{\mu}_{2}-\hat{\mu}_{1})\Sigma^{-1/2}\right)\Omega\right)\\
 & =\frac{1}{2}\mathrm{tr}\left(\Omega^{2}\right)\left\Vert H_{2}-H_{1}\right\Vert _{F}^{2}+\left\Vert (\hat{\mu}_{2}-\hat{\mu}_{1})\Sigma^{-1/2}\Omega^{1/2}\right\Vert _{F}^{2}
\end{align*}
where $\Omega\equiv\Sigma^{-1/2}\Sigma_{*}\Sigma^{-1/2}$ is a multivariate
generalization of the variance ratio $\sigma_{*}^{2}/\sigma^{2}$.

\end{figure*}

\begin{figure*}
\subsection*{Proof of Lemma \ref{lem:nonsharedComplexity}}

We first state the following definition from \cite{Galanti2008}.

\begin{definition}
Let $S_{1},S_{2}\subset\mathbb{R^{\mathbb{\mathrm{n}}}}$
be subspaces with $p_{1}=\mathrm{dim}(S_{1})\ge\mathrm{dim}(S_{2})=p_{2}\ge1.$
The principal angles $\theta_{k}\in\left[0,\pi/2\right]$ between
$S_{1}$ and $S_{2}$ are recursively defined for $k=1,...,p_{2}$
by 
\[
\cos(\theta_{k})=\underset{u\in S_{1}}{\max}\enskip\underset{v\in S_{2}}{\max}|u^{T}v|=u_{k}^{T}v_{k},||u||=||v||=1,
\]
subject to the constraints 
\[
u_{i}^{T}u=0,\enskip v_{i}^{T}v=0,\enskip i=1,...,k-1.
\]
\end{definition}

Now,
\begin{equation*}
    \mathrm{tr}(H_{1}-H_{2})^{2}=\kappa^2 \mathrm{tr}(P_{1}-P_{2})^{2}=\kappa^2
    \big(\mathrm{tr}(P_{1})+\mathrm{tr}(P_{2})-2\mathrm{tr}(P_{1}P_{2})\big)=\kappa^2\big(p_{1}+p_{2}-2\sum_{i=1}^{n}\lambda_{i}\big),
\end{equation*}
where $\lambda_{i}$ are the eigenvalues of $P_{1}P_{2}.$ Theorem
34 in \cite{Galanti2008} proves that $\mathrm{eig}(P_{1}P_{2})=(1_{s},\cos^{2}\theta_{s+i}\:(i=1,...,r),0_{n-s-r})$
where $1_s$ is an $s \times 1$ vector of ones, $0_{n-s-r}$ is an $(n-s-r) \times 1$ vector of zeroes,  $s=\dim(S_{1}\cap S_{2})$ is the number of $\theta_{j}$ which
are exactly zero, and $r$ is the number of $\theta_{j}$ in the open
interval $(0,\pi/2)$. Hence
\begin{equation*}
\mathrm{tr}(P_{1}-P_{2})^{2}=\kappa^2\bigg(p_{1}+p_{2}-2\big(s+\sum_{i=1}^{r}\cos^{2}(\theta_{k+i})\big)\bigg).
\end{equation*}

\end{figure*}

\acknow{Please include your acknowledgments here, set in a single paragraph. Please do not include any acknowledgments in the Supporting Information, or anywhere else in the manuscript.}


\bibliography{my_collection}

\begin{thebibliography}{10}

\bibitem{brooks2011handbook}
S Brooks, A Gelman, G Jones, XL Meng, {\em Handbook of {Markov Chain Monte
  Carlo}}.
\newblock (CRC press), (2011).

\bibitem{doucet2000sequential}
A Doucet, S Godsill, C Andrieu, On sequential {Monte} {Carlo} sampling methods
  for {Bayesian} filtering.
\newblock {\em\protect\JournalTitle{Statistics and computing}} \textbf{10},
  197--208 (2000).

\bibitem{griffiths2004finding}
TL Griffiths, M Steyvers, Finding scientific topics.
\newblock {\em\protect\JournalTitle{Proceedings of the National academy of
  Sciences}} \textbf{101}, 5228--5235 (2004).

\bibitem{Smets2007}
F Smets, R Wouters, {Shocks and frictions in US business cycles: A Bayesian
  DSGE approach}.
\newblock {\em\protect\JournalTitle{American Economic Review}} (2007).

\bibitem{stephan2009bayesian}
KE Stephan, WD Penny, J Daunizeau, RJ Moran, KJ Friston, Bayesian model
  selection for group studies.
\newblock {\em\protect\JournalTitle{Neuroimage}} \textbf{46}, 1004--1017
  (2009).

\bibitem{raftery1995bayesian}
AE Raftery, Bayesian model selection in social research.
\newblock {\em\protect\JournalTitle{Sociological methodology}} \textbf{25},
  111--164 (1995).

\bibitem{johnson2013revised}
VE Johnson, Revised standards for statistical evidence.
\newblock {\em\protect\JournalTitle{Proceedings of the National Academy of
  Sciences}} \textbf{110}, 19313--19317 (2013).

\bibitem{Hoeting1999}
JA Hoeting, D Madigan, AE Raftery, CT Volinsky, {Bayesian Model Averaging: A
  Tutorial}.
\newblock {\em\protect\JournalTitle{Statistical Science}} \textbf{14}, 382--417
  (1999).

\bibitem{Bernardo1994}
JM Bernardo, AF Smith, {\em {Bayesian Theory}}.
\newblock p. 604 (1994).

\bibitem{Gelfand1994}
AE Gelfand, DK Dey, {Bayesian model choice: asymptotics and exact
  calculations}.
\newblock {\em\protect\JournalTitle{Journal of the Royal Statistical Society
  Series B}} (1994).

\bibitem{geweke1999using}
J Geweke, Using simulation methods for {Bayesian} econometric models:
  inference, development, and communication.
\newblock {\em\protect\JournalTitle{Econometric reviews}} \textbf{18}, 1--73
  (1999).

\bibitem{Vehtari+Ojanen:2012}
A Vehtari, J Ojanen, A survey of {Bayesian} predictive methods for model
  assessment, selection and comparison.
\newblock {\em\protect\JournalTitle{Statistics Surveys}} \textbf{6}, 142--228
  (2012).

\bibitem{fong2019marginal}
E Fong, C Holmes, On the marginal likelihood and cross-validation.
\newblock {\em\protect\JournalTitle{arXiv preprint arXiv:1905.08737}} (2019).

\bibitem{Yao2018}
Y Yao, A Vehtari, D Simpson, A Gelman, Using stacking to average {B}ayesian
  predictive distributions (with discussion).
\newblock {\em\protect\JournalTitle{Bayesian Analysis}} \textbf{13}, 917--1003
  (2018).

\bibitem{li2019comparing}
M Li, DB Dunson, Comparing and weighting imperfect models using
  d-probabilities.
\newblock {\em\protect\JournalTitle{Journal of the American Statistical
  Association}}, 1--26 (2019).

\bibitem{Yang2018a}
Z Yang, T Zhu, {Bayesian selection of misspecified models is overconfident and
  may cause spurious posterior probabilities for phylogenetic trees}.
\newblock {\em\protect\JournalTitle{Proceedings of the National Academy of
  Sciences}} \textbf{115}, 1854--1859 (2018).

\bibitem{huggins2019using}
JH Huggins, JW Miller, Using bagged posteriors for robust inference and model
  criticism.
\newblock {\em\protect\JournalTitle{arXiv preprint arXiv:1912.07104}} (2019).

\bibitem{leff2008cortical}
AP Leff, et~al., The cortical dynamics of intelligible speech.
\newblock {\em\protect\JournalTitle{Journal of Neuroscience}} \textbf{28},
  13209--13215 (2008).

\bibitem{herbst2015bayesian}
EP Herbst, F Schorfheide, {\em Bayesian estimation of DSGE models}.
\newblock (Princeton University Press), (2015).

\bibitem{politis1994stationary}
DN Politis, JP Romano, The stationary bootstrap.
\newblock {\em\protect\JournalTitle{Journal of the American Statistical
  association}} \textbf{89}, 1303--1313 (1994).

\bibitem{friston2003dynamic}
KJ Friston, L Harrison, W Penny, Dynamic causal modelling.
\newblock {\em\protect\JournalTitle{Neuroimage}} \textbf{19}, 1273--1302
  (2003).

\bibitem{ashburner2014spm12}
J Ashburner, et~al., Spm12 manual.
\newblock {\em\protect\JournalTitle{Wellcome Trust Centre for Neuroimaging,
  London, UK}}, 2464 (2014).

\bibitem{Kass1995}
RE Kass, AE Raftery, {Bayes factors}.
\newblock {\em\protect\JournalTitle{Journal of the American Statistical
  Association}} (1995).

\bibitem{ruppert2003semiparametric}
D Ruppert, MP Wand, RJ Carroll, {\em Semiparametric regression}.
\newblock (Cambridge university press) No.{}~12, (2003).

\bibitem{Berk1966}
RH Berk, {Limiting Behavior of Posterior Distributions when the Model is
  Incorrect}.
\newblock {\em\protect\JournalTitle{The Annals of Mathematical Statistics}}
  \textbf{37}, 51--58 (1966).

\bibitem{fernandez2004comparing}
J Fern\'{a}ndez-Villaverde, JF Rubio-Ram\'{i}rez, Comparing dynamic equilibrium
  models to data: a bayesian approach.
\newblock {\em\protect\JournalTitle{Journal of Econometrics}} \textbf{123},
  153--187 (2004).

\bibitem{draper1995assessment}
D Draper, Assessment and propagation of model uncertainty.
\newblock {\em\protect\JournalTitle{Journal of the Royal Statistical Society:
  Series B (Methodological)}} \textbf{57}, 45--70 (1995).

\bibitem{Hastie2009}
T Hastie, R Tibshirani, J Friedman, {\em {The Elements of Statistical
  Learning}}.
\newblock (2009).

\bibitem{rasmussen2003gaussian}
CE Rasmussen, Gaussian processes in machine learning in {\em Summer School on
  Machine Learning}.
\newblock (Springer), pp. 63--71 (2003).

\bibitem{adolfson2007forecasting}
M Adolfson, J Lind{\'e}, M Villani, Forecasting performance of an open economy
  {DSGE} model.
\newblock {\em\protect\JournalTitle{Econometric Reviews}} \textbf{26}, 289--328
  (2007).

\bibitem{Geweke2012}
J Geweke, G Amisano, {Prediction with misspecified models}.
\newblock {\em\protect\JournalTitle{American Economic Review}} \textbf{102},
  482--486 (2012).

\bibitem{Geweke1999}
J Geweke, {Using Simulation Methods for Bayesian Econometric Models: Inference,
  Development and Communication}.
\newblock {\em\protect\JournalTitle{Econometric Revies}} \textbf{18}, 1--126
  (1999).

\bibitem{bao2010expectation}
Y Bao, A Ullah, Expectation of quadratic forms in normal and nonnormal
  variables with applications.
\newblock {\em\protect\JournalTitle{Journal of Statistical Planning and
  Inference}} \textbf{140}, 1193--1205 (2010).

\bibitem{cover2012elements}
TM Cover, JA Thomas, {\em Elements of information theory}.
\newblock (John Wiley \& Sons), (2012).

\bibitem{Galanti2008}
A Galanti, {Subspaces, angles and pairs of orthogonal projections}.
\newblock {\em\protect\JournalTitle{Linear and Multilinear Algebra}}
  \textbf{56}, 227--260 (2008).

\end{thebibliography}

\end{document}